\documentclass[leqno,a4paper,twoside]{article}%
\usepackage{amssymb}
\usepackage{amsfonts}
\usepackage{amsmath}
\usepackage{graphicx}%
\providecommand{\U}[1]{\protect\rule{.1in}{.1in}}
\newtheorem{theorem}{Theorem}

\newtheorem{condition}[theorem]{Condition}

\newtheorem{criterion}[theorem]{Criterion}
\newtheorem{definition}[theorem]{Definition}
\newtheorem{example}[theorem]{Example}

\newtheorem{proposition}[theorem]{Proposition}
\newtheorem{remark}[theorem]{Remark}

\newenvironment{proof}[1][Proof]{\noindent\textbf{#1.} }{\ \rule{0.5em}{0.5em}}
\begin{document}

\title{Approximate and Approximate Null-Controllability of a Class of Piecewise
Linear Markov Switch Systems\thanks{\textbf{AMS\ MSC:} 93B05, 93B25, 60J75,
\textbf{Keywords:} Approximate (null-)controllability \textperiodcentered
\ Controlled Markov switch process \textperiodcentered\ Invariance
\textperiodcentered\ Stochastic gene networks}%
\thanks{\textbf{Acknowledgement.} The work of the first author has been
partially supported by he French National Research Agency project PIECE,
number \textbf{ANR-12-JS01-0006.} The work of the third author has been
partially supported by the Grant PN-II-ID-PCE-2011-3-0843, no. 241/05.10.2011,
Deterministic and stochastic systems with state constraints.}}
\author{Dan Goreac\thanks{Universit\'{e} Paris-Est, LAMA (UMR 8050), UPEMLV, UPEC,
CNRS, F-77454, Marne-la-Vall\'{e}e, France, Dan.Goreac@u-pem.fr}
\and Alexandra Claudia Grosu\thanks{Faculty of Mathematics, \textquotedblleft
Alexandru Ioan Cuza\textquotedblright\ University, Bd. Carol I, no. 9-11,
Iasi, Romania}
\and Eduard-Paul Rotenstein\thanks{Faculty of Mathematics, \textquotedblleft
Alexandru Ioan Cuza\textquotedblright\ University, Bd. Carol I, no. 9-11,
Iasi, Romania}}
\maketitle
\maketitle

\begin{abstract}
We propose an explicit, easily-computable algebraic criterion for approximate
null-controllability of a class of general piecewise linear switch systems
with multiplicative noise. This gives an answer to the general problem left
open in \cite{GoreacMartinez2015}. The proof relies on recent results in
\cite{CFJ_2014} allowing to reduce the dual stochastic backward system to a
family of ordinary differential equations. Second, we prove by examples that
the notion of approximate controllability is strictly stronger than
approximate null-controllability. A sufficient criterion for this stronger
notion is also provided. The results are illustrated on a model derived from
repressed bacterium operon (given in \cite{Krishna29032005} and reduced in
\cite{crudu_debussche_radulescu_09}).

\end{abstract}

\section{Introduction}

This short paper aims at giving an answer to an approximate
(null-)controllability problem left open in \cite{GoreacMartinez2015}. We deal
with Markovian systems of switch type consisting of a couple mode/ trajectory
denoted by $\left(  \Gamma,X\right)  .$ The mode component $\Gamma$ evolves as
a pure jump Markov process and cannot be controlled. It corresponds to spikes
inducing regime switching. The second component $X$ obeys a controlled linear
stochastic differential equation (SDE) with respect to the compensated random
measure associated to $\Gamma$. The linear coefficients governing the dynamics
depend on the current mode.

The controllability problem deals with criteria allowing one to drive the
$X_{T}$ component arbitrarily close to acceptable targets. An extensive
literature on controllability is available in different frameworks:
finite-dimensional deterministic setting (Kalman's condition, Hautus test
\cite{Hautus}), infinite dimensional settings (via invariance criteria in
\cite{Schmidt_Stern_80}, \cite{Curtain_86}, \cite{russell_Weiss_1994},
\cite{Jacob_Zwart_2001}, \cite{Jacob_Partington_2006}, etc.), Brownian-driven
control systems (exact terminal-controllability in \cite{Peng_94}, approximate
controllability in \cite{Buckdahn_Quincampoix_Tessitore_2006}, \cite{G17},
mean-field Brownian-driven systems in \cite{G1}, infinite-dimensional setting
in \cite{Fernandez_Cara_Garrido_atienza_99}, \cite{Sarbu_Tessitore_2001},
\cite{Barbu_Rascanu_Tessitore_2003}, \cite{G16}, etc.), jump systems
(\cite{G10}, \cite{GoreacMartinez2015}, etc.). We refer to
\cite{GoreacMartinez2015} for more details on the literature as well as
applications one can address using switch models.

The paper \cite{GoreacMartinez2015} provides some necessary and some
sufficient conditions under which approximate controllability towards null
target can be achieved. In all generality, the conditions are either too
strong (sufficient) or too weak (only necessary). Equivalence is obtained in
\cite{GoreacMartinez2015} for particular cases : (i) Poisson-driven systems
with mode-independent coefficients and (ii) continuous switching. In the
present paper, we extend the work of \cite{GoreacMartinez2015} and give
explicit equivalence criterion for the general switching case. The approach
relies, in a first step, as it has already been the case in \cite[Theorem
1]{GoreacMartinez2015}, on duality techniques (briefly presented in Subsection
\ref{SubsectionDuality}). However, the intuition on this new criterion and its
proof are extensively based on the recent ideas in \cite{CFJ_2014}. The dual
backward stochastic system associated to controllability is interpreted as a
system of (backward) ordinary differential equations in Proposition
\ref{EquivalenceBSDEODE}. Reasoning on this new system provides the necessary
and sufficient criterion for approximate null-controllability for general
switching systems with mode-dependent multiplicative noise (Theorem
\ref{ThMain} whose proof relies on Propositions \ref{PropNec0Ctrl} and
\ref{PropSuff0Ctrl}). As a by-product, we considerably simplify the proofs of
\cite[Criteria 3 and 4]{GoreacMartinez2015} (in Subsection
\ref{SubsectionComparison}). Second, we give some elements on the stronger
notion of (general) approximate controllability. While the notions of
approximate and approximate null-controllability are known to coincide for
Poisson-driven systems with mode-independent coefficients, we give an example
(Example \ref{ExempleAppCtrl}) showing that this is no longer the case for
general switching systems. Furthermore, we show that the condition exhibited
in \cite[Proposition 3]{GoreacMartinez2015} in connection to approximate
null-controllability is actually sufficient for general approximate
controllability (see Condition \ref{SuffConditionAppCtrl}). The proof follows,
once again, from the deterministic reduction inspired by \cite{CFJ_2014}. The
theoretical results are illustrated on a model derived from repressed
bacterium operon (given in \cite{Krishna29032005} and reduced in
\cite{crudu_debussche_radulescu_09}).

We begin with presenting the problem, the standing assumptions and the main
results: the duality abstract characterization in Theorem \ref{dualityTh}, the
explicit criterion in Theorem \ref{ThMain}. We give a considerably simplified
proof of the results in \cite{GoreacMartinez2015} in Subsection
\ref{SubsectionComparison}. We discuss the difference between null and full
approximate controllability in Subsection \ref{SubsectionAppCtrl}, Example
\ref{ExempleAppCtrl} and give a sufficient criterion for the stronger notion
of approximate controllability (Criterion \ref{SuffConditionAppCtrl}). Section
\ref{SectionApplication} focuses on an example derived from
\cite{Krishna29032005} (see also \cite{crudu_debussche_radulescu_09}). The
proofs of the results and the technical constructions allowing to prove
Theorem \ref{ThMain} are gathered in Section \ref{SectionProofs}.

\section{The Control System and Main Results}

We briefly recall the construction of a particular class of pure jump, non
explosive processes on a space $\Omega$ and taking their values in a metric
space $\left(  E,\mathcal{B}\left(  E\right)  \right)  .$ Here, $\mathcal{B}%
\left(  E\right)  $ denotes the Borel $\sigma$-field of $E.$ The elements of
the space $E$ are referred to as modes. These elements can be found in
\cite{davis_93} in the particular case of piecewise deterministic Markov
processes (see also \cite{Bremaud_1981}). To simplify the arguments, we assume
that $E$ is finite and we let $p\geq1$ be its cardinal. The process is
completely described by a couple $\left(  \lambda,Q\right)  ,$ where
$\lambda:E\longrightarrow%
\mathbb{R}
_{+}$ and the measure $Q:E\longrightarrow\mathcal{P}\left(  E\right)  $, where
$\mathcal{P}\left(  E\right)  $ stands for the set of probability measures on
$\left(  E,\mathcal{B}\left(  E\right)  \right)  $ such that \ $Q\left(
\gamma,\left\{  \gamma\right\}  \right)  =0.$ Given an initial mode
$\gamma_{0}\in E,$ the first jump time satisfies $\mathbb{P}^{0,\gamma_{0}%
}\left(  T_{1}\geq t\right)  =\exp\left(  -t\lambda\left(  \gamma_{0}\right)
\right)  .$ The process $\Gamma_{t}:=\gamma_{0},$ on $t<T$ $_{1}.$ The
post-jump location $\gamma^{1}$ has $Q\left(  \gamma_{0},\cdot\right)  $ as
conditional distribution. Next, we select the inter-jump time $T_{2}-T_{1}$
such that $\mathbb{P}^{0,\gamma_{0}}\left(  T_{2}-T_{1}\geq t\text{ }/\text{
}T_{1},\gamma^{1}\right)  =\exp\left(  -t\lambda\left(  \gamma^{1}\right)
\right)  $ and set $\Gamma_{t}:=\gamma^{1},$ if $t\in\left[  T_{1}%
,T_{2}\right)  .$ The post-jump location $\gamma^{2}$ satisfies $\mathbb{P}%
^{0,\gamma_{0}}\left(  \gamma^{2}\in A\text{ }/\text{ }T_{2},T_{1},\gamma
^{1}\right)  =Q\left(  \gamma^{1},A\right)  ,$ for all Borel set $A\subset E.$
And so on. To simplify arguments on the equivalent ordinary differential
system, following \cite[Assumption (2.17)]{CFJ_2014}, we will assume that the
system stops after a non-random, fixed number $M>0$ of jumps i.e.
$\mathbb{P}^{0,\gamma_{0}}\left(  T_{M+1}=\infty\right)  =1$. The reader is
invited to note (see Remark \ref{RemarkM}) that, for large $M,$ the criteria
given in the main result (Theorem \ref{ThMain}) no longer depend on $M$ (due
to the finite dimension of the mode and state spaces).

We look at the process $\Gamma$ under $\mathbb{P}^{0,\gamma_{0}}$ and denote
by $\mathbb{F}^{0}$ the filtration $\left(  \mathcal{F}_{\left[  0,t\right]
}:=\sigma\left\{  \Gamma_{r}:r\in\left[  0,t\right]  \right\}  \right)
_{t\geq0}.$ The predictable $\sigma$-algebra will be denoted by $\mathcal{P}%
^{0}$ and the progressive $\sigma$-algebra by $Prog^{0}.$ As usual, we
introduce the random measure $q$ on $\Omega\times\left(  0,\infty\right)
\times E$ by setting $q\left(  \omega,A\right)  =\sum_{k\geq1}1_{\left(
T_{k}\left(  \omega\right)  ,\Gamma_{T_{k}\left(  \omega\right)  }\left(
\omega\right)  \right)  \in A},$ for all $\omega\in\Omega,$ $A\in
\mathcal{B}\left(  0,\infty\right)  \times\mathcal{B}\left(  E\right)  .$ The
compensated martingale measure is denoted by $\widetilde{q}$. (For our readers
familiar with \cite{GoreacMartinez2015}, we emphasize that the notation is
slightly different, the counting measure $q$ corresponds to $p$ in the cited
paper and the martingale measure $\widetilde{q}$ replaces $q$ in the same
reference. Further details on the compensator are given in Subsection
\ref{SubsTechnicalPrelim}.)

We consider a switch system given by a process $(X(t),\Gamma(t))$ on the state
space $%
\mathbb{R}
^{N}\times E,$ for some $N\geq1$ and the family of modes $E$. $\ $The control
state space is assumed to be some Euclidian space $%
\mathbb{R}
^{d},$ $d\geq1$. The component $X(t)$ follows a controlled differential system
depending on the hidden variable $\gamma$. We will deal with the following
model ($A$ is implicitly assumed to be $0$ after the last jump).%
\begin{equation}
dX_{s}^{x,u}=\left[  A\left(  \Gamma_{s}\right)  X_{s}^{x,u}+Bu_{s}\right]
ds+\int_{E}C\left(  \Gamma_{s-},\theta\right)  X_{s-}^{x,u}\widetilde{q}%
\left(  ds,d\theta\right)  ,\text{ }s\geq0,\text{ }X_{0}^{x,u}=x. \label{SDE0}%
\end{equation}
The operators $A\left(  \gamma\right)  \in%
\mathbb{R}
^{N\times N}$ , $B\in%
\mathbb{R}
^{N\times d}$ and $C\left(  \gamma,\theta\right)  \in%
\mathbb{R}
^{N\times N}$, for all $\gamma,\theta\in E$. For linear operators, we denote
by $\ker$ their kernel and by $\operatorname{Im}$ the image (or range) spaces.
Moreover, the control process $u:\Omega\times%
\mathbb{R}
_{+}\longrightarrow%
\mathbb{R}
^{d}$ is an $%
\mathbb{R}
^{d}$-valued, $\mathbb{F}^{0}-$ progressively measurable, locally square
integrable process. The space of all such processes will be denoted by
$\mathcal{U}_{ad}$ and referred to as the family of admissible control
processes. The explicit structure of such processes can be found in
\cite[Proposition 4.2.1]{Jacobsen}, for instance. Since the control process
does not (directly) intervene in the noise term, the solution of the above
system can be explicitly computed with $\mathcal{U}_{ad}$ processes instead of
the (more usual) predictable processes.

\subsection{The Duality Abstract Characterization of Approximate
Null-Controllability\label{SubsectionDuality}}

We begin with recalling the following approximate controllability concepts.

\begin{definition}
The system (\ref{SDE0}) is said to be approximately controllable in time $T>0$
starting from the initial mode $\gamma_{0}\in E,$ if, for every $\mathcal{F}%
_{\left[  0,T\right]  }$-measurable, square integrable $\xi\in\mathbb{L}%
^{2}\left(  \Omega,\mathcal{F}_{\left[  0,T\right]  },\mathbb{P}^{0,\gamma
_{0}};%
\mathbb{R}
^{N}\right)  $, every initial condition $x\in%
\mathbb{R}
^{N}$ and every $\varepsilon>0$, there exists some admissible control process
$u\in\mathcal{U}_{ad}$ such that $\mathbb{E}^{0,\gamma_{0}}\left[  \left\vert
X_{T}^{x,u}-\xi\right\vert ^{2}\right]  \leq\varepsilon.$ The system
(\ref{SDE0}) is said to be approximately null-controllable in time $T>0$ if
the previous condition holds for $\xi=0$ ($\mathbb{P}^{0,\gamma_{0}}$-a.s.).
\end{definition}

At this point, let us consider the backward (linear) stochastic differential
equation
\begin{equation}
\left\{
\begin{array}
[c]{l}%
dY_{t}^{T,\xi}=\left[  -A^{\ast}\left(  \Gamma_{t}\right)  Y_{t}^{T,\xi}%
-\int_{E}\left(  C^{\ast}\left(  \Gamma_{t},\theta\right)  +I\right)
Z_{t}^{T,\xi}\left(  \theta\right)  \lambda\left(  \Gamma_{t}\right)  Q\left(
\Gamma_{t},d\theta\right)  \right]  dt+\int_{E}Z_{t}^{T,\xi}\left(
\theta\right)  q\left(  dt,d\theta\right)  ,\\
Y_{T}^{T,\xi}=\xi\in\mathbb{L}^{2}\left(  \Omega,\mathcal{F}_{\left[
0,T\right]  },\mathbb{P}^{0,\gamma_{0}};%
\mathbb{R}
^{N}\right)  .
\end{array}
\right.  \label{BSDE0}%
\end{equation}
Classical arguments on the controllability operators and the duality between
the concepts of controllability and observability lead to the following
characterization (cf. \cite[Theorem 1]{GoreacMartinez2015}).

\begin{theorem}
[{\cite[Theorem 1]{GoreacMartinez2015}}]\label{dualityTh}The necessary and
sufficient condition for approximate null-controllability (resp. approximate
controllability) of (\ref{SDE0}) with initial mode $\gamma_{0}\in E$ is that
any solution $\left(  Y_{t}^{T,\xi},Z_{t}^{T,\xi}\left(  \cdot\right)
\right)  $ of the dual system (\ref{BSDE0}) for which $Y_{t}^{T,\xi}\in\ker
B^{\ast}$ $,$ $\mathbb{P}^{0,\gamma_{0}}\mathbb{\otimes}Leb$ almost everywhere
on $\Omega\times\left[  0,T\right]  $ should equally satisfy $Y_{0}^{T,\xi
}=0,$ $\mathbb{P}^{0,\gamma_{0}}-$almost surely (resp. $Y_{t}^{T,\xi}=0,$
$\mathbb{P}^{0,\gamma_{0}}\mathbb{\otimes}Leb-a.s.$).
\end{theorem}

\begin{remark}
Concerning the operator $A,$ it is assumed to be a switched matrix but it
could also depend on $\left(  t,\Gamma_{t}\right)  $ or on all the times and
marks prior to $t.$ This is why, we implicitly assumed that $A=0$ after the
last jump ($M^{th}$) occurs. Similar assertions are true for $C$ (otherwise,
the backward equation (\ref{BSDE0}) should be written with the compensator
$\widehat{q}$ replacing $\lambda\left(  \Gamma_{t}\right)  Q\left(  \Gamma
_{t},d\theta\right)  $.) The reader may also look at the end of Subsection
\ref{SubsTechnicalPrelim}.
\end{remark}

\subsection{Main Result : An Iterative Invariance Criterion}

Before stating the main result of our paper, we need the following invariance
concepts (cf. \cite{Curtain_86}, \cite{Schmidt_Stern_80}).

\begin{definition}
We consider a linear operator $\mathcal{A\in}%
\mathbb{R}
^{N\times N}$ and a family $\mathcal{C=}\left(  \mathcal{C}_{i}\right)
_{1\leq i\leq k}\subset%
\mathbb{R}
^{N\times N}$.

(i) A set $V\subset%
\mathbb{R}
^{N}$ is said to be $\mathcal{A}$- invariant if $\mathcal{A}V\subset V.$

(ii) A set $V\subset%
\mathbb{R}
^{N}$ is said to be $\left(  \mathcal{A};\mathcal{C}\right)  $- invariant if
$\mathcal{A}V\subset V+%
{\textstyle\sum\limits_{i=1}^{k}}
\operatorname{Im}\mathcal{C}_{i}.$
\end{definition}

We construct a mode-indexed family of linear subspaces of $%
\mathbb{R}
^{N}$ denoted by $\left(  V_{\gamma}^{M,n}\right)  _{0\leq n\leq M,\text{
}\gamma\in E}$ by setting
\begin{equation}
\mathcal{A}^{\ast}\left(  \gamma\right)  :=A^{\ast}\left(  \gamma\right)
-\int_{E}\left(  C^{\ast}\left(  \gamma,\theta\right)  +I\right)
\lambda(\gamma)Q(\gamma,d\theta)\text{ and }V_{\gamma}^{M,M}=\ker B^{\ast},
\label{CalA}%
\end{equation}
for all $\gamma\in E,$ and computing, for every $0\leq n\leq M-1,$
\begin{equation}
\left.  V_{\gamma}^{M,n}\text{ the largest }\left(  \mathcal{A}^{\ast}\left(
\gamma\right)  ;\left[  \left(  C^{\ast}(\gamma,\theta)+I\right)
\Pi_{V_{\theta}^{M,n+1}}:\theta\in E,\text{ }Q\left(  \gamma,\theta\right)
>0\right]  \right)  -\text{invariant subspace of }\ker B^{\ast}.\right.
\label{V^n_spaces}%
\end{equation}
Here, $\Pi_{V}$ denotes the orthogonal projection operator onto the linear
space $V\subset%
\mathbb{R}
^{N}$. Whenever there is no confusion at risk, having fixed the maximal number
of jumps $M\geq1,$ we drop the dependency on $M$ (i.e. we write $V_{\gamma
}^{n}$ instead of $V_{\gamma}^{M,n}$ for all $0\leq n\leq M$).

\begin{remark}
\label{RemarkM}(i) A simple recurrence argument shows that $V_{\gamma}%
^{M,n}\subset V_{\gamma}^{M,m}$, for every $0\leq n\leq m\leq M$. Furthermore,
$V_{\gamma}^{M,M-n}=V_{\gamma}^{M^{\prime},M^{\prime}-n},$ for all $0\leq
n\leq M\leq M^{\prime}.$ Moreover, since the dimension of $\ker B^{\ast}$
cannot exceed $N,$ $V_{\gamma}^{M,0}=V_{\gamma}^{\min\left(  M,N^{p}\right)
,0}.$

(ii) This spaces do not depend on the choice of the controllability horizon
$T>0.$ Therefore, if the approximate (null-)controllability is described by
these sets, it is independent of the time horizon.
\end{remark}

The main result of the paper is the following.

\begin{theorem}
\label{ThMain}The switch system (\ref{SDE0}) is approximately
null-controllable (in time $T>0$) with $\gamma_{0}$ as initial mode, if and
only if the generated set $V_{\gamma_{0}}^{0}$ reduces to $\left\{  0\right\}
.$
\end{theorem}

The proof is postponed to Section \ref{SectionProofs}. This proof uses the
reduction of backward equations with respect to Marked point processes to a
system of ordinary differential equations given in \cite{CFJ_2014}. In order
to formulate this system (see Proposition \ref{EquivalenceBSDEODE}), we need
to explain some concepts and notations in Subsection \ref{SubsTechnicalPrelim}%
. To prove necessity of the condition, one uses convenient feedback controls
and the equivalence between invariance and the concept of feedback invariance
(see Proposition \ref{PropNec0Ctrl}). Sufficiency (given by Proposition
\ref{PropSuff0Ctrl}) follows from (time-) invariance of convenient linear
subspaces with respect to ordinary differential dynamics.

\subsection{Comparison With \cite{GoreacMartinez2015}%
\label{SubsectionComparison}}

We begin with giving a different (and simpler) proof of (some of) the results
in \cite{GoreacMartinez2015}. Besides the general (abstract) characterization
of approximate and approximate null-controllability, explicit invariance
criteria were given in two specific settings.

(i) In \textbf{the case without multiplicative noise} $C=0,$ one notes that
the subspaces $V_{\gamma}^{n}$ (for $0\leq n<M$) do not depend on $n.$ They
reduce, in fact, to the largest $\mathcal{A}^{\ast}\left(  \gamma\right)
$-invariant subspace of $\ker B^{\ast}.$ Moreover, in this framework,
$\mathcal{A}^{\ast}\left(  \gamma\right)  $-invariance and $A^{\ast}\left(
\gamma\right)  $-invariance coincide and Theorem \ref{ThMain} yields the following.

\begin{criterion}
[{\cite[Criterion 4]{GoreacMartinez2015}}]The system (\ref{SDE0}) is
approximately null-controllable (with initial mode $\gamma_{0}\in E$) if and
only if the largest subspace of $\ker B^{\ast}$ which is $A^{\ast}\left(
\gamma_{0}\right)  $ - invariant is reduced to the trivial subspace $\left\{
0\right\}  $ for all $\gamma_{0}\in E.$
\end{criterion}

(ii) In \textbf{the case of Poisson-driven systems with mode-independent
coefficients} $A$ and $C,$ one works with the mode-independent operator
$\mathcal{A}^{\ast}:=A^{\ast}-\int_{E}\left(  C^{\ast}\left(  \theta\right)
+I\right)  \lambda Q(d\theta)$. The reader familiar with \cite[Criterion
3]{GoreacMartinez2015} will note that the necessary and sufficient criterion
concerns a notion of strict invariance. We get the same condition provided the
system has the possibility to stabilize (the maximal number of jumps $M\geq
N+1$ is allowed to exceed the dimension of the state space). Moreover, without
loss of generality, one assumes that $E$ is the support of $Q.$

\begin{criterion}
[{\cite[Criterion 3]{GoreacMartinez2015}}]Let us assume that $A\in%
\mathbb{R}
^{N\times N},$ $B\in%
\mathbb{R}
^{N\times d}$ are fixed and $C\left(  \theta\right)  \in%
\mathbb{R}
^{N\times N},$ for all $\theta\in E$ and that $\lambda\left(  \gamma\right)
Q\left(  \gamma,d\theta\right)  $ is independent of $\gamma\in E.$ Moreover,
we assume that $M\geq N+1$. Then the associated system is approximately
null-controllable if and only if the largest subspace $V_{0}\subset\ker
B^{\ast}$ which is $\left(  A^{\ast};\left[  C^{\ast}\left(  \theta\right)
\Pi_{V_{0}}:\theta\in E\right]  \right)  $-invariant is reduced to $\left\{
0\right\}  $.
\end{criterion}

\begin{proof}
The reader will note that the $V^{n}$ spaces in (\ref{V^n_spaces}) no longer
depend on $\gamma\in E.$ They are obviously non-decreasing in $n$ (see Remark
\ref{RemarkM}). Since $\ker B^{\ast}\subset\mathbb{R}^{N},$ it follows that,
provided that $M\geq N+1,$ one has $V^{0}=V^{1}$ (indeed, $V^{0}\subset
V^{1}\subset...\subset V^{M}=\ker B^{\ast}$ and, whenever $V^{k}=V^{k+1},$ for
some $0\leq k\leq M-1$, it follows (by definition), that $V^{j}=V^{k},$ for
all $j\leq k.$ On the other hand, the inclusions cannot always be strict if
$M\geq N+1$ by recalling that the dimension of $\ker B^{\ast}$ cannot exceed
$N$). Moreover, this space $V^{0}$ is the largest subspace $V_{0}\subset\ker
B^{\ast}$ which is $\left(  \mathcal{A}^{\ast};\left[  \left(  C^{\ast}\left(
\theta\right)  +I\right)  \Pi_{V_{0}}:\theta\in E\right]  \right)  $-invariant
which is the same as $\left(  A^{\ast};\left[  C^{\ast}\left(  \theta\right)
\Pi_{V_{0}}:\theta\in E\right]  \right)  $-invariant. The proof is complete by
invoking Theorem \ref{ThMain}.
\end{proof}

\subsection{Approximate or Approximate
Null-Controllability\label{SubsectionAppCtrl}}

Using Riccati techniques, one proves (see \cite[Criterion 3]%
{GoreacMartinez2015}) that, for Poisson-driven systems with mode-independent
coefficients, approximate controllability and approximate null-controllability
properties coincide. However, in the case of actual switching systems, the two
notions have no reason to and do not coincide. This is illustrated by the
following example.

\begin{example}
\label{ExempleAppCtrl}We consider the space dimension $N=4,$ the control
dimension $d=2,$ $E=\left\{  0,1\right\}  $, a switching rate $\lambda=1$ and
a transition probability $Q\left(  \gamma,1-\gamma\right)  =1,$ for $\gamma\in
E.$ Moreover, we consider, for $\gamma\in\left\{  0,1\right\}  ,$
\[
B=\left(
\begin{array}
[c]{cc}%
1 & 0\\
0 & 1\\
0 & 0\\
0 & 0
\end{array}
\right)  ,\text{ }A\left(  \gamma\right)  =\left(
\begin{array}
[c]{cccc}%
0 & 0 & 0 & 0\\
0 & 0 & 0 & 0\\
1+\gamma & 0 & 0 & 0\\
0 & 2-\gamma & 0 & 0
\end{array}
\right)  ,\text{ }C\left(  \gamma,1-\gamma\right)  =\left(
\begin{array}
[c]{cccc}%
-1 & 0 & 0 & 0\\
0 & -1 & 0 & 0\\
\gamma & 0 & -1 & 0\\
0 & 1-\gamma & 0 & -1
\end{array}
\right)  .
\]
The reader is invited to note that $\ker B^{\ast}=span\left\{  e^{3}%
,e^{4}\right\}  $ (standard vectors of the basis of $%
\mathbb{R}
^{4}$)$.$ Thus, simple computations yield $V_{0}^{1}\subset span\left\{
e^{4}\right\}  ,$ $V_{1}^{1}\subset span\left\{  e^{3}\right\}  .$ Hence,
$V_{0}^{0}=V_{1}^{0}=\left\{  0\right\}  $ and the system is approximately
null-controllable starting from every initial mode (if $M\geq2$). However, if
one considers $\gamma_{0}=0,$ assumes the mode can jump twice $M=2$ and sets
$\xi:=1_{T_{1}\leq T<T_{2}}e^{3}-1_{T_{2}\leq T}e^{3},$ then one easily notes
that $\left(  Y_{t},Z_{t}\right)  :=\left(  1_{T_{1}\leq t\leq T,t<T_{2}}%
e^{3}-1_{T_{2}\leq t\leq T}e^{3},\left(  1_{t\leq T\wedge T_{1}}%
-2\times1_{T_{1}<t\leq T}\right)  e^{3}\right)  $ obey the equation
(\ref{BSDE0}). To this purpose, it suffices to note that $A^{\ast}(\Gamma
_{t})Y_{t}+\left(  C^{\ast}\left(  \Gamma_{t},1-\Gamma_{t}\right)  +I\right)
Z_{t}=0$ on $\left[  0,T\wedge T_{2}\right]  .$ For every $u\in\mathcal{U}%
_{ad},$ It\^{o}'s formula (e.g. \cite[Chapter II, Section 5, Theorem
5.1]{Ikeda_Watanabe_1981}) applied to the inner product $\left\langle
X_{\cdot}^{0,u},Y_{\cdot}\right\rangle $ on $\left[  0,T\right]  $ yields
$\mathbb{E}^{0,0}\left[  \left\langle X_{T}^{0,u},\xi\right\rangle \right]
=\mathbb{E}^{0,\gamma_{0}}\left[  \int_{0}^{T}\left\langle u_{t},B^{\ast}%
Y_{t}\right\rangle dt\right]  =0.$ In particular, this implies that
$\mathbb{E}^{0,0}\left[  \left\vert X_{T}^{0,u}-\xi\right\vert ^{2}\right]
\geq\mathbb{E}^{0,0}\left[  \left\vert \xi\right\vert ^{2}\right]  >0$ and,
thus, the system (\ref{SDE0}) is not approximately controllable (towards $\xi$).
\end{example}

In fact, the reader may note that the null-controllability property strongly
depends on the initial mode (through the computation of $V_{\gamma_{0}}^{0}$
as last step). A sufficient criterion (already available in
\cite{GoreacMartinez2015}) is that the largest subspace of $\ker B^{\ast}$
which is $\left(  \mathcal{A}^{\ast}\left(  \gamma_{0}\right)  ;\left[
\left(  C^{\ast}(\gamma_{0},\theta)+I\right)  \Pi_{\ker B^{\ast}}:\theta\in
E,Q\left(  \gamma_{0},\theta\right)  >0\right]  \right)  -$invariant should be
reduced to $\left\{  0\right\}  .$ It turns out that asking this condition to
hold true for all $\gamma_{0}\in E$ actually implies approximate
controllability. (The proof is postponed to Section \ref{SectionProofs}.)

\begin{condition}
\label{SuffConditionAppCtrl}Let us assume that the largest $\left(
\mathcal{A}^{\ast}\left(  \gamma\right)  ;\left[  \left(  C^{\ast}%
(\gamma,\theta)+I\right)  \Pi_{\ker B^{\ast}}:Q\left(  \gamma,\theta\right)
>0\right]  \right)  $-invariant subspace of $\ker B^{\ast}$ is reduced to
$\left\{  0\right\}  $, for every $\gamma\in E.$ Then, for every $T>0$ and
every $\gamma_{0}\in E$, the system (\ref{SDE0}) is approximately controllable
in time $T>0.$
\end{condition}

\begin{remark}
\label{RemNotNec}The reader is invited to note that the notion of $\left(
\mathcal{A}^{\ast}\left(  \gamma\right)  ;\left[  \left(  C^{\ast}%
(\gamma,\theta)+I\right)  \Pi_{\ker B^{\ast}}:Q\left(  \gamma,\theta\right)
>0\right]  \right)  $ -invariance and that of $\ \left(  A^{\ast}\left(
\gamma\right)  ;\left[  \left(  C^{\ast}(\gamma,\theta)+I\right)  \Pi_{\ker
B^{\ast}}:Q\left(  \gamma,\theta\right)  >0\right]  \right)  $ -invariance
coincide for subspaces of $\ker B^{\ast}$. Second, according to
\cite[Criterion 3]{GoreacMartinez2015}, the notions of approximate and
approximate null-controllability coincide in the context of Poisson-driven
systems with mode-independent coefficients. Then, a careful look at
\cite[Example 4]{GoreacMartinez2015} provides an example of system which is
approximately controllable without satisfying the sufficient condition given before.
\end{remark}

\section{Towards Applications\label{SectionApplication}}

\textbf{A model.} We will explain how the previous method can be applied in
the study of stochastic gene networks. To this purpose, we consider the
following reaction system describing a repressed bacterium operon model
introduced in \cite{Krishna29032005}.%
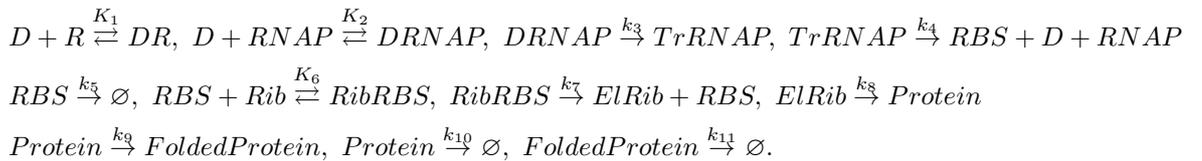
\begin{align*}
&  D+R\overset{K_{1}}{\rightleftarrows}DR,\text{ }D+RNAP\overset{K_{2}%
}{\rightleftarrows}DRNAP,\text{ }DRNAP\overset{k_{3}}{\rightarrow
}TrRNAP,\text{ }TrRNAP\overset{k_{4}}{\rightarrow}RBS+D+RNAP\\
&  RBS\overset{k_{5}}{\rightarrow}\varnothing,\text{ }RBS+Rib\overset{K_{6}%
}{\rightleftarrows}RibRBS,\text{ }RibRBS\overset{k_{7}}{\rightarrow
}ElRib+RBS,\text{ }ElRib\overset{k_{8}}{\rightarrow}Protein\\
&  Protein\overset{k_{9}}{\rightarrow}FoldedProtein,\text{ }%
Protein\overset{k_{10}}{\rightarrow}\varnothing,\text{ }%
FoldedProtein\overset{k_{11}}{\rightarrow}\varnothing.
\end{align*}

\textbf{Partitioning and simplifying.} The authors of
\cite{crudu_debussche_radulescu_09} propose a partition of "species" according
to which only $ElRib,Protein$ and $FoldedProtein$ are continuous. The
averaging procedures in \cite[Figure 4]{crudu_debussche_radulescu_09} simplify
the model to%
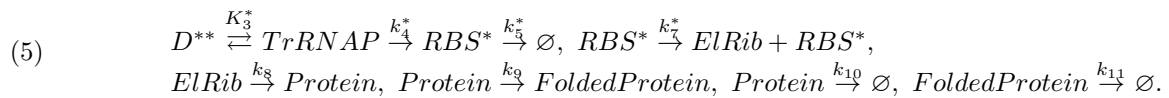
\begin{equation}
\left.
\begin{array}
[c]{l}%
D^{\ast\ast}\overset{K_{3}^{\ast}}{\rightleftarrows}TrRNAP\overset{k_{4}%
^{\ast}}{\rightarrow}RBS^{\ast}\overset{k_{5}^{\ast}}{\rightarrow}%
\varnothing,\text{ }RBS^{\ast}\overset{k_{7}^{\ast}}{\rightarrow
}ElRib+RBS^{\ast},\text{ }\\
ElRib\overset{k_{8}}{\rightarrow}Protein,\text{ }Protein\overset{k_{9}%
}{\rightarrow}FoldedProtein,\text{ }Protein\overset{k_{10}}{\rightarrow
}\varnothing,\text{ }FoldedProtein\overset{k_{11}}{\rightarrow}\varnothing.
\end{array}
\right.  \label{BacteriumOperon}%
\end{equation}
Due to the conservation law of $\left[  D,R,DR,RNAP,DRNAP,TrRNAP\right]  $ one
should have something like $D^{\ast\ast}+TrRNAP\simeq1.$

It is known (\cite[Page 21]{crudu_debussche_radulescu_09}) that "$RBS^{\ast}$
presents infrequent bursts of activity leading to rapid production of $ElRib$"
and "$RBS^{\ast}$ rapidly switches to $0$ by the reaction $RBS^{\ast
}\rightarrow\varnothing"$. To take into account these elements and keep the
conservation law, we proceed as follows :

(1) as $RBS^{\ast}$ switches to $0,$ $D^{\ast\ast}$ will be reset to $1$
(hence, $D^{\ast\ast}+TrRNAP+RBS^{\ast}=1)$;

(2) bursts (given by the reaction having $k_{7}^{\ast}$ as speed) will be
considered as part of the stochastic updating of the continuous species and
will have null-mean (i.e. they will multiply the martingale measure generated
by the mode switching mechanism). In our toy-model, as $RBS^{\ast}$ switches
to $1$, stochastic bursts on $ElRib$ will affect (in multiplicative way) the
synthesis of $Protein$ (i.e. the reaction $ElRib\overset{k_{8}}{\rightarrow
}Protein$)$.$

\textbf{A toy mathematical system.} The first condition leads to a mode space
$E=\left\{  e^{1},e^{2},e^{3}\right\}  $ consisting of the standard vector
basis of $%
\mathbb{R}
^{3},$ with a jump intensity $\lambda$ and a transition measure $\left(
Q\left(  e^{i},\left\{  e^{j}\right\}  \right)  =Q_{i,j}\right)  _{1\leq
i,j\leq3}$ given by
\begin{equation}
\lambda\left(  \gamma\right)  =\left\langle \left(
\begin{array}
[c]{c}%
k_{3}^{\ast}\\
k_{-3}^{\ast}+k_{4}^{\ast}\\
k_{5}^{\ast}%
\end{array}
\right)  ,\gamma\right\rangle >0,\text{ for all }\gamma\in E,\text{ }Q=\left(
\begin{array}
[c]{ccc}%
0 & 1 & 0\\
\frac{k_{-3}^{\ast}}{k_{-3}^{\ast}+k_{4}^{\ast}} & 0 & \frac{k_{4}^{\ast}%
}{k_{-3}^{\ast}+k_{4}^{\ast}}\\
1 & 0 & 0
\end{array}
\right)  . \label{ExempleIntensityQ}%
\end{equation}
We are going to assume that the positive reaction speeds $k_{7}^{\ast}%
,k_{8},k_{9}$ and $k_{11}$ depend on the mode $\gamma$ (note that $RBS^{\ast}$
is part of $\gamma$ and intervenes to get $ElRib)$ and, maybe, of external
one-dimensional control parameters (temperature or catalysts). Since all the
reactions concerning continuous components have one reactant, the resulting
ODE will be linear (see \cite[Eq. (28)]{crudu_debussche_radulescu_09}). A
first order model for the control will give $dx_{t}=\left[  A\left(
\Gamma_{t}\right)  x_{t}+Bu_{t}\right]  dt,$ where $A$ is given by
(\ref{ExempleCoeffABC}). Furthermore, in our toy model, let us assume that the
external control focuses on regulation of $ElRib$ (i.e. $B=\left(
\begin{array}
[c]{ccc}%
1 & 0 & 0
\end{array}
\right)  ^{t}=e^{1}$). We add to that the bursts (see item (2) above) to
finally get a (toy-)model of type (\ref{SDE0}) for which, for every
$\gamma,\theta\in E,$
\begin{equation}
\left.  B=e^{1},\text{ }A\left(  \gamma\right)  =\left(
\begin{array}
[c]{ccc}%
-k_{8}\left(  \gamma\right)  & 0 & 0\\
k_{8}\left(  \gamma_{t}\right)  & -k_{9}\left(  \gamma\right)  & 0\\
0 & k_{9}\left(  \gamma_{t}\right)  & -k_{11}\left(  \gamma\right)
\end{array}
\right)  ,\text{ }C\left(  \gamma,\theta\right)  =\left(
\begin{array}
[c]{ccc}%
0 & 0 & 0\\
k_{7}^{\ast}\left(  \gamma\right)  & 0 & 0\\
0 & 0 & 0
\end{array}
\right)  ,\text{ }k_{7}^{\ast}\left(  \gamma\right)  =1_{e^{3}}\left(
\gamma\right)  .\right.  \label{ExempleCoeffABC}%
\end{equation}

\textbf{Approximate null-controllability.} The largest subspace of $\ker
B^{\ast}$ which is $\left(  A^{\ast}\left(  e^{2}\right)  -\lambda\left(
e^{2}\right)  I;\Pi_{\ker B^{\ast}}\right)  $-invariant reduces to
$span\left(  e^{3}\right)  $ and the largest subspace of $\ker B^{\ast}$ which
is $\left(  A^{\ast}\left(  e^{1}\right)  -\lambda\left(  e^{1}\right)
I;\Pi_{span\left(  e^{3}\right)  }\right)  $invariant is $\left\{  0\right\}
$ (recall that $k_{8}$ and $k_{9}$ are reaction speeds and, thus, are strictly
positive and so is $\lambda$)$.$ Due to the structure of the transition
measure $Q$, as soon as $M\geq2$, the system is approximately
null-controllable starting from $e^{1}.$ Nevertheless, the space $\ker
B^{\ast}\ $being $\mathcal{A}^{\ast}\left(  e^{3}\right)  -\left(
k_{8}\left(  e^{3}\right)  -k_{5}^{\ast}\right)  \left(  C^{\ast}\left(
e^{3},e^{1}\right)  +I\right)  $-invariant, constructions similar to Example
\ref{ExempleAppCtrl} show that, provided $e^{3}$ is reachable in $M$ jumps,
the system is not approximately controllable.

\section{Proof of the Results\label{SectionProofs}}

\subsection{Technical Preliminaries\label{SubsTechnicalPrelim}}

Before giving the reduction of our backward stochastic equation to a system of
ODE, we need to introduce some notations making clear the stochastic structure
of several concepts : final data, predictable and c\`{a}dl\`{a}g adapted
processes and compensator of the initial random measure. The notations in this
subsection follow the ordinary differential approach from \cite{CFJ_2014}.
Since we are only interested in what happens on $\left[  0,T\right]  ,$ we
introduce a cemetery state $\left(  \infty,\overline{\gamma}\right)  $ which
will incorporate all the information after $T\wedge T_{M}.$ It is clear that
the conditional law of $T_{n+1}$ given $\left(  T_{n},\Gamma_{T_{n}}\right)  $
is now composed by an exponential part on $\left[  T_{n}\wedge T,T\right]  $
and an atom at $\infty.$ Similarly, the conditional law of $\Gamma_{T_{n+1}}$
given $\left(  T_{n+1},T_{n},\Gamma_{T_{n}}\right)  $ is the Dirac mass at
$\overline{\gamma}$ if $T_{n+1}=\infty$ and given by $Q$ otherwise. Finally,
under the assumption $\mathbb{P}^{0,\gamma_{0}}\left(  T_{M+1}=\infty\right)
=1$, after $T_{M},$ the marked point process is concentrated at the cemetery state.

We set $\overline{E}_{T}:\mathcal{=}\left(  \left[  0,T\right]  \times
E\right)  \cup\left\{  \left(  \infty,\overline{\gamma}\right)  \right\}  $.
For every $n\geq1,$ we let $\overline{E}_{T,n}\subset\left(  \overline{E}%
_{T}\right)  ^{n+1}$ be the set of all marks of type $e=\left(  \left(
t_{0},\gamma_{0}\right)  ,...,\left(  t_{n},\gamma_{n}\right)  \right)  ,$
where%
\begin{equation}%
\begin{array}
[c]{l}%
t_{0}=0,\text{ }\left(  t_{i}\right)  _{0\leq i\leq n}\text{ are
non-decreasing; }t_{i}<t_{i+1},\text{ if }t_{i}\leq T\text{; }\left(
t_{i},\gamma_{i}\right)  =\left(  \infty,\overline{\gamma}\right)  \text{, if
}t_{i}>T,\text{ }\forall0\leq i\leq n-1,\text{ }%
\end{array}
\label{Def_e_Trajectory}%
\end{equation}
and endow it with the family of all Borel sets $\mathcal{B}_{n}$. For these
sequences, the maximal time is denoted by $\left\vert e\right\vert :=t_{n}$.
Moreover, by abuse of notation, we set $\gamma_{\left\vert e\right\vert
}:=\gamma_{n}.$ Whenever $T\geq t>\left\vert e\right\vert ,$ we set
\begin{equation}
e\oplus\left(  t,\gamma\right)  :=\left(  \left(  t_{0},\gamma_{0}\right)
,...,\left(  t_{n},\gamma_{n}\right)  ,\left(  t,\gamma\right)  \right)
\in\overline{E}_{T,n+1}.\label{Def_e_Concatenation}%
\end{equation}
By defining
\begin{equation}
e_{n}:=\left(  \left(  0,\gamma_{0}\right)  ,\left(  T_{1},\Gamma_{T_{1}%
}\right)  ,...,\left(  T_{n},\Gamma_{T_{n}}\right)  \right)
,\label{en_RandomVar}%
\end{equation}
we get an $\overline{E}_{T,n}-$valued random variable, corresponding to our
mode trajectories.

\textbf{The final data} $\xi$ is $\mathcal{F}_{\left[  0,T\right]  }%
-$measurable and, thus, for every $n\geq0,$ there exists a $\mathcal{B}%
_{n}/\mathcal{B}\left(
\mathbb{R}
^{N}\right)  -$measurable function $\overline{E}_{T,n}\ni e\mapsto\xi
^{n}\left(  e\right)  \in%
\mathbb{R}
^{N}$ such that:%
\begin{equation}
\text{If }\left\vert e\right\vert =\infty,\text{ then }\xi^{n}\left(
e\right)  =0.\text{ Otherwise, on }T_{n}\left(  \omega\right)  \leq
T<T_{n+1}\left(  \omega\right)  ,\text{ }\xi\left(  \omega\right)  =\xi
^{n}\left(  e_{n}\left(  \omega\right)  \right)  .\text{ }%
\label{Def_RandomVar}%
\end{equation}

\textbf{A c\`{a}dl\`{a}g process} $Y$ \textbf{continuous except, maybe, at
switching times} $T_{n}$ is given by the existence of a family of
$\mathcal{B}_{n}\otimes\mathcal{B}\left(  \left[  0,T\right]  \right)
/\mathcal{B}\left(
\mathbb{R}
^{N}\right)  $-measurable functions $y^{n}$ such that, for all $e\in
\overline{E}_{T,n},$ $y^{n}\left(  e,\cdot\right)  $ is continuous on $\left[
0,T\right]  $ and constant on $\left[  0,T\wedge\left\vert e\right\vert
\right]  $ and
\begin{equation}
\text{If }\left\vert e\right\vert =\infty,\text{ then }y^{n}\left(
e,\cdot\right)  =0.\text{ Otherwise, on }T_{n}\left(  \omega\right)  \leq
t<T_{n+1}\left(  \omega\right)  ,\text{ }Y_{t}\left(  \omega\right)
=y^{n}\left(  e_{n}\left(  \omega\right)  ,t\right)  ,\text{ }t\leq
T\text{.}\label{Def_Cadlag}%
\end{equation}

Similar, an $%
\mathbb{R}
^{N}-$valued $\mathbb{F}$-\textbf{predictable process} $Z$ defined on
$\Omega\times\left[  0,T\right]  \times E$ is given by the existence of a
family of $\mathcal{B}_{n}\otimes\mathcal{B}\left(  \left[  0,T\right]
\right)  \otimes\mathcal{B}\left(  E\right)  /\mathcal{B}\left(
\mathbb{R}
^{N}\right)  -$measurable functions $z^{n}$ satisfying%
\begin{equation}
\text{If }\left\vert e\right\vert =\infty,\text{ then }z^{n}\left(
e,\cdot,\cdot\right)  =0.\text{ On }T_{n}\left(  \omega\right)  <t\leq
T_{n+1}\left(  \omega\right)  ,\text{ }Z_{t}\left(  \omega,\gamma\right)
=z^{n}\left(  e_{n}\left(  \omega\right)  ,t,\gamma\right)  ,\text{ for }t\leq
T\text{, }\gamma\in E\text{.}\label{Def_Predictable}%
\end{equation}

To deduce the form of \textbf{the compensator}, one simply writes
$\widehat{q}\left(  \omega,dt,d\gamma\right)  :=%
{\displaystyle\sum\limits_{n\geq0}}
\widehat{q}_{e_{n}\left(  \omega\right)  }^{n}\left(  dt,d\gamma\right)
1_{T_{n}\left(  \omega\right)  <t\leq T_{n+1}\left(  \omega\right)  \wedge T}$
such that%
\begin{equation}
\left\{
\begin{array}
[c]{l}%
\text{If }n\geq M,\text{ then }\widehat{q}_{e}^{n}\left(  dt,d\gamma\right)
=\delta_{\overline{\gamma}}\left(  d\gamma\right)  \delta_{\infty}\left(
dt\right)  \smallskip\text{. If }n\leq M-1,\smallskip\\
\widehat{q}_{e}^{n}\left(  dt,d\gamma\right)  :=\lambda(\gamma_{\left\vert
e\right\vert })Q(\gamma_{\left\vert e\right\vert },d\gamma)1_{\left\vert
e\right\vert <\infty,t\in\left[  \left\vert e\right\vert ,T\right]
}Leb\left(  dt\right)  +\delta_{\overline{\gamma}}\left(  d\gamma\right)
\delta_{\infty}\left(  dt\right)  1_{\left(  \left\vert e\right\vert
<\infty,t>T\right)  \cup\left\vert e\right\vert =\infty},\smallskip
\end{array}
\right.  . \label{Def_Compensator}%
\end{equation}

Let us now concentrate on the specific form of the jump contribution $Z$ (to
the BSDE (\ref{BSDE0})). We consider a c\`{a}dl\`{a}g process $Y$ continuous
except, maybe, at switching times $T_{n}$. Then, as explained before, this can
be identified with a family $\left(  y^{n}\right)  .$ We construct, for every
$n\geq0,$%
\begin{equation}
\widehat{y}^{n+1}\left(  e,t,\gamma\right)  :=y^{n+1}\left(  e\oplus\left(
t,\gamma\right)  ,t\right)  1_{\left\vert e\right\vert <t} \label{yhatn+1}%
\end{equation}
and $Y_{T_{n+1}}$ can be obtained by simple integration of the previous
quantity with respect to the conditional law of $\left(  T_{n+1}%
,\Gamma_{T_{n+1}}\right)  $ knowing $\mathcal{F}_{T_{n}}.$ Then, $Z$ is given
by $z^{n}\left(  e,t,\gamma\right)  :=\widehat{y}^{n+1}\left(  e,t,\gamma
\right)  -y^{n}\left(  e,t\right)  .$

The \textbf{coefficient} function\textbf{ }$A\left(  \Gamma_{t}\right)  $ is
adapted and can be seen as follows: if $\left\vert e\right\vert =\infty,$ then
$A=0;$ otherwise, one works with $A\left(  \gamma_{\left\vert e\right\vert
}\right)  .$ Similar constructions hold true for $C.$ In fact, the results of
the present paper can be generalized to more general path-dependence of the coefficients.

\subsection{Reduction to a System of Linear ODEs}

We consider the family of (ordinary) differential equations%
\begin{equation}
\left\{
\begin{array}
[c]{l}%
y^{M}\left(  e_{M}\left(  \omega\right)  ,\cdot\right)  =\xi^{M}\left(
e_{M}\left(  \omega\right)  \right)  \text{. For }n\leq M-1\smallskip,\text{
}y^{n}\left(  e_{n}\left(  \omega\right)  ,T\right)  =\xi^{n}\left(
e_{n}\left(  \omega\right)  \right)  ,\\
dy^{n}\left(  e_{n}\left(  \omega\right)  ,t\right)  =-A^{\ast}\left(
\gamma_{\left\vert e_{n}\left(  \omega\right)  \right\vert }\right)
y^{n}\left(  e_{n}\left(  \omega\right)  ,t\right)  dt\\
\text{ \ \ \ \ \ \ \ \ \ \ \ \ \ \ \ \ \ \ \ \ \ }-\int_{E}\left(  C^{\ast
}\left(  \gamma_{\left\vert e_{n}\left(  \omega\right)  \right\vert }%
,\theta\right)  +I\right)  \left(  \widehat{y}^{n+1}\left(  e_{n}\left(
\omega\right)  ,t,\theta\right)  -y^{n}\left(  e_{n}\left(  \omega\right)
,t\right)  \right)  \widehat{q}_{e_{n}\left(  \omega\right)  }^{n}\left(
dt,d\theta\right)  \\
(=-\mathcal{A}^{\ast}\left(  \gamma_{\left\vert e_{n}\left(  \omega\right)
\right\vert }\right)  y^{n}\left(  e_{n}\left(  \omega\right)  ,t\right)  dt\\
-\sum_{\theta\in E}\lambda(\gamma_{\left\vert e_{n}\left(  \omega\right)
\right\vert })Q(\gamma_{\left\vert e_{n}\left(  \omega\right)  \right\vert
},\theta)\left(  C^{\ast}\left(  \gamma_{\left\vert e_{n}\left(
\omega\right)  \right\vert },\theta\right)  +I\right)  y^{n+1}\left(
e_{n}\left(  \omega\right)  \oplus\left(  t,\theta\right)  ,t\right)  dt),
\end{array}
\right.  \label{SystemODE}%
\end{equation}

where we have used the notation (\ref{CalA}). The following result adapts
\cite[Lemma 7]{CFJ_2014} to our case.

\begin{proposition}
\label{EquivalenceBSDEODE}A c\`{a}dl\`{a}g adapted process $Y$ given by a
family of functions $\left(  y^{n}\right)  $ as in (\ref{Def_Cadlag}) is
solution to (\ref{BSDE0}) if and only if, for $\mathbb{P}$-almost all $\omega$
and all $0\leq n\leq M,$ it satisfies the system (\ref{SystemODE}).
\end{proposition}

The proof is quasi-identical to the one of \cite[Lemma 7]{CFJ_2014}. The only
difference in our case is the presence of the term $-A^{\ast}\left(
\gamma_{\left\vert e_{n}\left(  \omega\right)  \right\vert }\right)
y^{n}\left(  e_{n}\left(  \omega\right)  ,t\right)  dt$ which is, of course,
classical. The results of \cite[Lemma 7]{CFJ_2014} apply directly if one
assumes that $\lambda(\gamma)$ $>0$ for all $\gamma\in E$ (that is if there
exists no absorbing state). Otherwise, we actually get an ODE of type
$dy^{n}\left(  e_{n}\left(  \omega\right)  ,t\right)  =-A^{\ast}\left(
\gamma_{\left\vert e_{n}\left(  \omega\right)  \right\vert }\right)
y^{n}\left(  e_{n}\left(  \omega\right)  ,t\right)  dt.$

\subsection{An Iterative Invariance-Based Criterion (Proof of Theorem
\ref{ThMain})}

As already hinted in \cite{GoreacMartinez2015}, the (approximate)
controllability properties can be expressed with respect to invariance
conditions. The equivalence between the dual (backward) stochastic equation
(\ref{BSDE0}) and the (backward) ordinary differential system (\ref{SystemODE}%
) yields the following approximate controllability criterion.

\begin{proposition}
\label{PropNec0Ctrl}If the system (\ref{SDE0}) is approximately
null-controllable with $\gamma_{0}$ as initial mode, then the generated set
$V_{\gamma_{0}}^{0}$ reduces to $\left\{  0\right\}  .$
\end{proposition}

\begin{proof}
Using classical results on the different notions of invariance (e.g.
\cite[Theorem 3.2]{Schmidt_Stern_80}, see also \cite[Lemma 4.6]{Curtain_86}),
invariance is equivalent to feedback invariance. Thus, one gets the existence
of a family of operators $F_{\gamma,\theta}^{n}\in\mathcal{L}\left(
V_{\gamma}^{n};V_{\theta}^{n+1}\right)  $ such that $\left(  \mathcal{A}%
^{\ast}\left(  \gamma\right)  +\sum_{_{\theta\in E,\text{ }Q\left(
\gamma,\theta\right)  >0}}\left(  C^{\ast}(\gamma,\theta)+I\right)
F_{\gamma,\theta}^{n}\right)  V_{\gamma}^{n}\subset V_{\gamma}^{n},$ for all
$n\geq0.$ We begin with picking (an arbitrary) $v_{0}\in V_{\gamma_{0}}^{0}$
and define $\zeta^{0}\left(  t_{0},\gamma_{0}\right)  =v_{0}.$ We proceed by
setting, for every $n\geq1$ and $e_{n}\in\overline{E}_{T,n}$, $\phi_{n,e_{n}}$
to be the unique solution of the ordinary differential system%
\[
\left\{
\begin{array}
[c]{l}%
d\phi_{n,e_{n}}\left(  t\right)  =-\left(  \mathcal{A}^{\ast}\left(
\gamma_{\left\vert e_{n}\right\vert }\right)  +\sum_{_{\theta\in E\text{,
}Q\left(  \gamma^{n},\theta\right)  >0}}\left(  C^{\ast}(\gamma_{\left\vert
e_{n}\right\vert },\theta)+I\right)  F_{\gamma_{\left\vert e_{n}\right\vert
},\theta}^{n}\right)  \phi_{n,e_{n}}\left(  t\right)  dt,\text{ }\left\vert
e_{n}\right\vert \leq t\leq T\\
\phi_{n,e_{n}}\left(  \left\vert e_{n}\right\vert \right)  =\left\{
\begin{array}
[c]{l}%
\zeta^{n}\left(  e_{n}\right)  ,\text{ if }\left\vert e_{n}\right\vert
<\infty\text{,}\\
0,\text{ otherwise.}%
\end{array}
\right.  \text{ and }\\
\zeta^{n+1}\left(  e_{n}\oplus\left(  t,\theta\right)  \right)  =\left\{
\begin{array}
[c]{l}%
\frac{1}{\lambda(\gamma_{\left\vert e_{n}\right\vert })Q(\gamma_{\left\vert
e_{n}\right\vert },\theta)}F_{\gamma_{\left\vert e_{n}\right\vert },\theta
}^{n}\phi_{n}\left(  e_{n},t\right)  1_{T\geq t>\left\vert e_{n}\right\vert
},\text{ if }\lambda(\gamma_{\left\vert e_{n}\right\vert })Q(\gamma
_{\left\vert e_{n}\right\vert },\theta)>0,\\
0,\text{ otherwise}%
\end{array}
\right.
\end{array}
\right.
\]
One also sets $\phi_{n,e_{n}}\left(  t\right)  =\phi_{n,e_{n}}\left(
\left\vert e_{n}\right\vert \vee t\right)  $ to extend the solution for
$t\in\left[  0,T\right]  $. Then, one easily notes that $\phi_{n,e_{n}}\left(
t\right)  \in\ker B^{\ast},$ for all $1\leq n\leq M$, all $e_{n}\in
\overline{E}_{T,n}$ and all $t\in\left[  0,T\right]  .$ Moreover, a simple
glance at the construction shows that by setting $y^{n}\left(  e_{n},t\right)
:=\phi_{n,e_{n}}\left(  t\right)  $, for $1\leq n\leq M$, all $e_{n}%
\in\overline{E}_{T,n}$ and all $t\in\left[  0,T\right]  ,$ one gets the
solution of (\ref{SystemODE}) with the particular choice of the final data
$\xi$ such that $\xi^{n}\left(  e_{n}\right)  =\phi_{n,e_{n}}\left(  T\right)
.$ Since we have assumed the system (\ref{SDE0}) to be approximately
null-controllable, Theorem \ref{dualityTh} and Proposition
\ref{EquivalenceBSDEODE} yield $v_{0}=0.$ The proof is complete by recalling
that $v_{0}\in V_{\gamma_{0}}^{0}$ is arbitrary.
\end{proof}

\bigskip At this point, the reader may want to note that these considerations
involve one equation at the time. The invariant space obtained is then
employed for the next equation and gives a coherent character to the system.
The basic idea is to provide some kind of local in time invariance of the sets
concerned. In \cite{GoreacMartinez2015}, this is done using Riccati
techniques. But, except for special cases, the solvability of these stochastic
schemes is far from obvious. Due to the ordinary differential structure of the
equivalent system (\ref{SystemODE}), we are able to elude these techniques and
work directly on the deterministic systems.

\begin{proposition}
\label{PropSuff0Ctrl}Conversely, if the generated set $V_{\gamma_{0}}^{0}$
reduces to $\left\{  0\right\}  ,$ then the system (\ref{SDE0}) is
approximately null-controllable with $\gamma_{0}$ as initial mode.
\end{proposition}

\begin{proof}
We begin with a solution of (\ref{BSDE0}) for which $Y$ belongs to $\ker
B^{\ast}.$ We prove by descending recurrence that $y^{n}\left(  e,t\right)
\in V_{\gamma_{\left\vert e\right\vert }}^{n},$ for all $t\in\left[
0,T\right]  $ and all $e\in\overline{E}_{T,n}$ (starting from $\gamma_{0}$),
where we use the structure (\ref{Def_Cadlag})$.$ The assertion is obvious for
$n=M$ since, by notation, $V_{\cdot}^{M}=\ker B^{\ast}.$ We assume it to hold
true for $n+1\leq M$ and prove it for $n\geq0.$ By equation (\ref{SystemODE}),
one has
\[
dy^{n}\left(  e,t\right)  =\left(  -\mathcal{A}^{\ast}\left(  \gamma
_{\left\vert e\right\vert }\right)  y^{n}\left(  e,t\right)  -\sum_{\theta\in
E}\lambda(\gamma_{\left\vert e\right\vert })Q(\gamma_{\left\vert e\right\vert
},\theta)\left(  C^{\ast}\left(  \gamma_{\left\vert e\right\vert }%
,\theta\right)  +I\right)  y^{n+1}\left(  e\oplus\left(  t,\theta\right)
,t\right)  \right)  dt.
\]
We have assumed that $y^{n}\left(  e,t\right)  \in\ker B^{\ast}$ and, thus,
$\left[  I-\Pi_{\ker B^{\ast}}\right]  y^{n}\left(  e,t\right)  =0.$ We infer
that
\[
\mathcal{A}^{\ast}\left(  \gamma_{\left\vert e\right\vert }\right)
y^{n}\left(  e,t\right)  +\sum_{\theta\in E}\lambda(\gamma_{\left\vert
e\right\vert })Q(\gamma_{\left\vert e\right\vert },\theta)\left(  C^{\ast
}\left(  \gamma_{\left\vert e\right\vert },\theta\right)  +I\right)
y^{n+1}\left(  e\oplus\left(  t,\theta\right)  ,t\right)  \in\ker B^{\ast}.
\]
Hence, using the recurrence assumption, $y^{n}\left(  e,t\right)  $ is (for
almost all $t\in\left[  0,T\right]  ),$ an element of the linear space
\[
W^{0}:=\left\{
\begin{array}
[c]{c}%
v\in\ker B^{\ast}:\exists w^{\theta}\in V_{\theta}^{n+1},\text{ for all
}\theta\in E\text{ s.t. }Q\left(  \gamma_{\left\vert e\right\vert }%
,\theta\right)  >0\text{ satisfying}\\
\mathcal{A}^{\ast}\left(  \gamma_{\left\vert e\right\vert }\right)
v+\underset{\theta\in E,\text{ }Q\left(  \gamma_{\left\vert e\right\vert
},\theta\right)  >0}{\sum}\left(  C^{\ast}\left(  \gamma_{\left\vert
e\right\vert },\theta\right)  +I\right)  w^{\theta}\in\ker B^{\ast}%
\end{array}
\right\}  .
\]
By repeating our argument, we prove that $y^{n}\left(  e,t\right)  $ is (for
almost all $t\in\left[  0,T\right]  ),$ an element of the linear space%
\[
W^{m+1}:=\left\{
\begin{array}
[c]{c}%
v\in W^{m}:\exists w^{\theta}\in V_{\theta}^{n+1},\text{ for all }\theta\in
E\text{ s.t. }Q\left(  \gamma_{\left\vert e\right\vert },\theta\right)
>0\text{ satisfying}\\
\mathcal{A}^{\ast}\left(  \gamma_{\left\vert e\right\vert }\right)
v+\underset{\theta\in E,\text{ }Q\left(  \gamma_{\left\vert e\right\vert
},\theta\right)  >0}{\sum}\left(  C^{\ast}\left(  \gamma_{\left\vert
e\right\vert },\theta\right)  +I\right)  w^{\theta}\in W^{m}%
\end{array}
\right\}  ,
\]
for every $m\geq0.$ Then, $W:=\underset{0\leq m\leq N}{\cap}W^{m}$ is an
$\left(  \mathcal{A}^{\ast}\left(  \gamma_{\left\vert e\right\vert }\right)
;\left[  \left(  C^{\ast}\left(  \gamma_{\left\vert e\right\vert }%
,\theta\right)  +I\right)  \Pi_{V_{\theta}^{n+1}}:Q\left(  \gamma_{\left\vert
e\right\vert },\theta\right)  >0\right]  \right)  -$invariant subspace of the
(at most $N$-dimensional) space $\ker B^{\ast}.$ Therefore, we have proven
that $y^{n}\left(  e,t\right)  \in$ $V_{\gamma_{\left\vert e\right\vert }}%
^{n}.$ To complete our argument, one only needs to recall that, by assumption,
$V_{\gamma_{0}}^{0}=\left\{  0\right\}  $ and use Theorem \ref{dualityTh} and
Proposition \ref{EquivalenceBSDEODE}.
\end{proof}

\subsection{Proof of Sufficiency Condition \ref{SuffConditionAppCtrl} for
Approximate Controllability}

\begin{proof}
[Proof of Condition \ref{SuffConditionAppCtrl}]In light of the Theorem
\cite[Theorem 1]{GoreacMartinez2015} and Proposition \ref{EquivalenceBSDEODE},
one only needs to show that the only solution of (\ref{SystemODE}) remaining
in $\ker B^{\ast}$ is constant $0.$ One proceeds as in the Proof of
Proposition \ref{PropSuff0Ctrl} starting with a solution of (\ref{BSDE0}) for
which $Y$ belongs to $\ker B^{\ast}$ and showing that $y^{n}\left(
e,t\right)  \in$ $V_{\gamma_{\left\vert e\right\vert }}^{n}\subset\ker
B^{\ast},$ for all $t\in\left[  0,T\right]  $ (recall that $y^{n}$ is
continuous). One recalls that $V_{\gamma}^{n}$ is $\left(  \mathcal{A}^{\ast
}\left(  \gamma\right)  ;\left[  \left(  C^{\ast}(\gamma,\theta)+I\right)
\Pi_{V_{\theta}^{n+1}}:Q\left(  \gamma,\theta\right)  >0\right]  \right)
$-invariant, for every $\gamma\in E$. Hence, a fortiori, $V_{\gamma
_{\left\vert e\right\vert }}^{n}$ is $\left(  \mathcal{A}^{\ast}\left(
\gamma_{\left\vert e\right\vert }\right)  ;\left[  \left(  C^{\ast}%
(\gamma_{\left\vert e\right\vert },\theta)+I\right)  \Pi_{\ker B^{\ast}%
}:Q\left(  \gamma_{\left\vert e\right\vert },\theta\right)  >0\right]
\right)  $-invariant. Our assumption implies that $V_{\gamma_{\left\vert
e\right\vert }}^{n}=\left\{  0\right\}  $ and approximate controllability follows.
\end{proof}

\bibliographystyle{plain}
\bibliography{bibliografie_17042016}

\def\cprime{$'$}
\begin{thebibliography}{10}

\bibitem{Barbu_Rascanu_Tessitore_2003}
V.~Barbu, A.~R{\u{a}}{\c{s}}canu, and G.~Tessitore.
\newblock Carleman estimates and controllability of linear stochastic heat
  equations.
\newblock {\em Appl. Math. Optim.}, 47(2):97--120, 2003.

\bibitem{Bremaud_1981}
Pierre Br{\'e}maud.
\newblock {\em Point processes and queues : martingale dynamics}.
\newblock Springer series in statistics. Springer-Verlag, New York, 1981.

\bibitem{Buckdahn_Quincampoix_Tessitore_2006}
R.~Buckdahn, M.~Quincampoix, and G.~Tessitore.
\newblock A characterization of approximately controllable linear stochastic
  differential equations.
\newblock In {\em Stochastic partial differential equations and
  applications---{VII}}, volume 245 of {\em Lect. Notes Pure Appl. Math.},
  pages 53--60. Chapman \& Hall/CRC, Boca Raton, FL, 2006.

\bibitem{CFJ_2014}
F.~Confortola, M.~Fuhrman, and J.~Jacod.
\newblock Backward stochastic differential equations driven by a marked point
  process: an elementary approach, with an application to optimal control.
\newblock {\em Annals of Applied Probability}, to appear, 2015.
\newblock arXiv:1407.0876.

\bibitem{crudu_debussche_radulescu_09}
A.~Crudu, A.~Debussche, and O.~Radulescu.
\newblock Hybrid stochastic simplifications for multiscale gene networks.
\newblock {\em BMC Systems Biology}, page 3:89, 2009.

\bibitem{Curtain_86}
R.~F. Curtain.
\newblock {Invariance concepts in infinite dimensions}.
\newblock {\em {SIAM J. Control and Optim.}}, {24}({5}):{1009--1030}, {SEP}
  {1986}.

\bibitem{davis_93}
M.~H.~A. Davis.
\newblock {\em {M}arkov models and optimization}, volume~49 of {\em Monographs
  on Statistics and Applied Probability}.
\newblock Chapman \& Hall, London, 1993.

\bibitem{Fernandez_Cara_Garrido_atienza_99}
E.~Fern{\'a}ndez-Cara, M.~J. Garrido-Atienza, and J.~Real.
\newblock On the approximate controllability of a stochastic parabolic equation
  with a multiplicative noise.
\newblock {\em C. R. Acad. Sci. Paris S\'er. I Math.}, 328(8):675--680, 1999.

\bibitem{G17}
D.~Goreac.
\newblock A {K}alman-type condition for stochastic approximate controllability.
\newblock {\em Comptes Rendus Mathematique}, 346(3--4):183 -- 188, 2008.

\bibitem{G16}
D.~Goreac.
\newblock Approximate controllability for linear stochastic differential
  equations in infinite dimensions.
\newblock {\em Applied Mathematics and Optimization}, 60(1):105--132, 2009.

\bibitem{G10}
D.~Goreac.
\newblock A note on the controllability of jump diffusions with linear
  coefficients.
\newblock {\em IMA Journal of Mathematical Control and Information},
  29(3):427--435, 2012.

\bibitem{G1}
D.~Goreac.
\newblock Controllability properties of linear mean-field stochastic systems.
\newblock {\em Stochastic Analysis and Applications}, 32(02):280--297, 2014.

\bibitem{GoreacMartinez2015}
D.~Goreac and M.~Martinez.
\newblock Algebraic invariance conditions in the study of approximate (null-)
  controllability of {M}arkov switch processes.
\newblock {\em Math. Control Signals Systems}, 27(4):551--578, 2015.

\bibitem{Hautus}
M.~L.~J. Hautus.
\newblock Controllability and observability conditions of linear autonomous
  systems.
\newblock {\em Nederl. Akad. Wetensch. Proc. Ser. A 72 Indag. Math.},
  31:443--448, 1969.

\bibitem{Ikeda_Watanabe_1981}
N.~Ikeda and S.~Watanabe.
\newblock {\em Stochastic Differential Equations and Diffusion Processes},
  volume~24 of {\em North-Holland Mathematical Library}.
\newblock North-Holland Publishing Co., Amsterdam--New York; Kodansha, Ltd.,
  Tokyo, 1981.

\bibitem{Jacob_Partington_2006}
B.~Jacob and J.~R. Partington.
\newblock On controllability of diagonal systems with one-dimensional input
  space.
\newblock {\em {Systems \& Control Letters}}, 55(4):321 -- 328, 2006.

\bibitem{Jacob_Zwart_2001}
B.~Jacob and H.~Zwart.
\newblock Exact observability of diagonal systems with a finite-dimensional
  output operator.
\newblock {\em {Systems \& Control Letters}}, 43(2):101 -- 109, 2001.

\bibitem{Jacobsen}
Martin Jacobsen.
\newblock {\em Point Process Theory And Applications. Marked Point and
  Piecewise Deterministic Processes}.
\newblock Birkh{\"a}user Verlag GmbH, 2006.

\bibitem{Krishna29032005}
Sandeep Krishna, Bidisha Banerjee, T.~V. Ramakrishnan, and G.~V. Shivashankar.
\newblock Stochastic simulations of the origins and implications of long-tailed
  distributions in gene expression.
\newblock {\em Proceedings of the National Academy of Sciences of the United
  States of America}, 102(13):4771--4776, 2005.

\bibitem{Peng_94}
S.~Peng.
\newblock Backward stochastic differential equation and exact controllability
  of stochastic control systems.
\newblock {\em Progr. Natur. Sci. (English Ed.)}, 4:274--284, 1994.

\bibitem{russell_Weiss_1994}
D.~Russell and G.~Weiss.
\newblock A general necessary condition for exact observability.
\newblock {\em SIAM Journal on Control and Optimization}, 32(1):1--23, 1994.

\bibitem{Schmidt_Stern_80}
E.~J. P.~G. Schmidt and R.~J. Stern.
\newblock Invariance theory for infinite dimensional linear control systems.
\newblock {\em {Applied Mathematics and Optimization}}, {6}({2}):{113--122},
  {1980}.

\bibitem{Sarbu_Tessitore_2001}
M.~Sirbu and G.~Tessitore.
\newblock {Null controllability of an infinite dimensional SDE with state- and
  control-dependent noise}.
\newblock {\em {Systems \& Control Letters}}, {44}({5}):{385--394}, {DEC 14}
  {2001}.

\end{thebibliography}

\end{document}